\documentclass[11pt]{amsart}
\usepackage{amsmath,amsthm, amscd, amssymb, amsfonts}
\usepackage[dvips, dvipsnames, usenames]{color}

\usepackage{graphicx} 

\usepackage{mathrsfs} 
\usepackage{figsize}

\usepackage[ruled,vlined]{algorithm2e}
\usepackage{algorithmicx}

\usepackage{float} 

\usepackage{tikz}

\usepackage{tkz-graph}
\usetikzlibrary{arrows}
\usetikzlibrary{calc}
\usetikzlibrary{matrix}
\usetikzlibrary{decorations.pathreplacing}  
\usepgflibrary{shapes.geometric} 
\usetikzlibrary[topaths]
\usetikzlibrary{positioning}
\usetikzlibrary{arrows}  
\usepackage{mathrsfs}
\usepackage{tkz-berge}

\usepackage{float}
\usepackage{faktor}
\usepackage{tabularx}

\theoremstyle{plain}
\newtheorem{theorem}{Theorem}[section]
\newtheorem{corollary}[theorem]{Corollary}
\newtheorem{proposition}[theorem]{Proposition}
\newtheorem{lemma}[theorem]{Lemma}

\newtheorem{remark}[theorem]{Remark}
\newtheorem{example}[theorem]{Example}

\usepackage[colorlinks]{hyperref}
\hypersetup{
    colorlinks=true,
    linkcolor=blue,
    citecolor=blue,
    urlcolor=cyan,
}

\numberwithin{equation}{section}

\def\N{\mathbb{N}}

\def\Z{\mathbb{Z}}

\def\F{\mathbb{F}} 

\def\frakN{\mathfrak{N}}

\def\fix{\mathrm{fix}}
\def\supp{\mathrm{supp}}
\def\ord{\mathrm{ord}}

\def\psl{\mathrm{PSL}}

\def\nil{\mathrm{nil}} 
  
\def\stab{\mathrm{Stab}}
\def\fix{\mathrm{Fix}}
\def\orb{\mathrm{Orb}}

\title[On the nilpotent graph of a finite group]{On the nilpotent graph of a finite group}

\thanks{...}

\author{Jaime Torres}
 \address{Universidad del Norte,  Departamento de Matem\'aticas y Estad\'istica, Km 5 via a Puerto Colombia, Barranquilla, Colombia.}
  \email{atjaime@uninorte.edu.co}

\author{Ismael Gutierrez}
 \address{Universidad del Norte,  Departamento de Matem\'aticas y Estad\'istica, Km 5 via a Puerto Colombia, Barranquilla, Colombia.}
  \email{isgutier@uninorte.edu.co
 }

\author{E. J. Garc\'ia-Claro}
 \address{Departamento de Matemáticas, Universidad Autónoma Metropolitana, Unidad Iztapalapa, Código postal 09340, Ciudad de México - México
 }
  \email{eliasjaviergarcia@gmail.com 
 }


\begin{document}

\begin{abstract}
If $G$ is a non-nilpotent group and $\nil(G)= \{g\in G:  \langle g, h\rangle \ \text{is nilpotent for all} \ h\in G\}$, the nilpotent graph of $G$ is the graph with set of vertices $G - \nil(G)$  in which two distinct vertices are related if they generate a nilpotent subgroup of $G$. Several properties of the nilpotent graph associated with a finite non-nilpotent group $G$ are studied in this work. Lower bounds for the clique number and the number of connected components of the nilpotent graph of $G$ are presented in terms of the size of its Fitting subgroup and the number of its strongly self-centralizing subgroups, respectively. It is proved the nilpotent graph of the symmetric group of degree $n$ is disconnected if and only if $n$ or $n-1$ is a prime number, and no finite non-nilpotent group has a self-complementary nilpotent graph. For the dihedral group $D_{n}$, it is determined the number of connected components of its nilpotent graph is one more than $n$ when $n$ is odd; or one more than the $2'$-part of $n$ when $n$ is even. In addition, a formula for the number of connected components of the nilpotent graph of $\psl(2,q)$, where $q$ is a prime power, is provided. Finally, necessary and sufficient conditions for specific subsets of a group, containing connected components of its nilpotent graph, to contain one of its Sylow $p$-subgroups are studied; and it is shown the nilpotent graph of a finite non-nilpotent group $G$ with $\nil(G)$ of even order is non-Eulerian. 

\end{abstract}

\maketitle


\section*{Introduction}\label{Section}

 Applying techniques and concepts from graph theory in the study of group theory has been a topic of growing interest in the last decades. This approach provides a visual perspective that can shed light on the group's properties in question. An outstanding example of this connection between groups and graphs is the non-commutative graph associated with a non-commutative group $G$, introduced by A. Abdollahi et al., in \cite{Abdollahi1}. This is the graph with set of vertices $G - Z(G)$  in which two distinct vertices are related if they do not commute. This graph captures several commutativity relations between the group elements and reveals interesting patterns in its inner structure. For example, if $G$ is a finite non-abelian nilpotent group and $H$ is a finite non-abelian group such that $|G| = |H|$ and the non-commutative graph of $H$ is isomorphic to the non-commutative graph of $G$, then $H$ is nilpotent. Later, A.K. Das et al. studied the genus of the commutative graph \cite{Das1}, which is precisely the complement of the non-commutative graph, i.e., the same set of vertices is considered, but two different vertices are related if they commute. That paper proved that the order of a finite non-abelian group is bounded by a function that depends on the genus of its respective commutative graph. 

A few years later, A. Abdollahi et al. introduced in \cite{Abdollahi2} the non-nilpotent graph associated with a group $G$, where the vertex set is $G - \nil(G)$ and in which two distinct vertices are related if they generate a non-nilpotent subgroup of $G$. An interesting fact about this graph is that, for finite non-nilpotent groups, the graph is always connected. A generalization of this graph, involving formations of finite groups, was presented by A. Lucchini et al. in \cite{Lucchini2}. Sometime later, as in the case of the non-commutative graph, a research was carried out on the genus of the nilpotent graph (the complement of the non-nilpotent graph) \cite{Das2}. 


This work studies further properties of the nilpotent graphs associated with finite non-nilpotent groups. The paper is organized as follows: Section \ref{Section1} introduces several necessary concepts and results from graph and group theory that will be used later. In Section \ref{Section2}, a Sagemath algorithm to compute the nilpotent graph of a non-nilpotent group is provided along with some examples of nilpotent graphs that will be useful to illustrate later results. It is also proved that the nilpotent graph of a finite group $G$ is bipartite if and only if $G$ is isomorphic to the symmetric group of degree three, among other results. Section \ref{Section3} establishes several lower bounds on the clique number and the number of connected components of the nilpotent graph of $G$ in terms of the size of its Fitting subgroup and the number of its strongly self-centralizing subgroups, respectively. In particular, it is shown the nilpotent graph of $S_{n}$ (for $n\geq 5$) is disconnected if and only if $n$ or $n-1$ is a prime number. In Section \ref{Section4}, sufficient conditions to guarantee the connectedness of the nilpotent graph are provided. In addition, it is proved that the nilpotent graph of a finite group is connected if and only if the nilpotent graph of its direct product with any finite nilpotent group is connected. It is also proved that the nilpotent graph of any finite non-nilpotent group is never self-complementary (among other results). In Sections \ref{Section5} and \ref{Section6}, it is demonstrated that the nilpotent graphs of the dihedral group $D_n$ and the projective special linear groups $\psl(2, q)$ where $q$ is a prime power are always disconnected by offering formulas for the number of the connected components of their respective graphs. Finally, necessary and sufficient conditions for certain subsets of a group (that contain connected components of its nilpotent graph) to contain one of its Sylow $p$-subgroups are studied, and it is shown that the nilpotent graph of a finite non-nilpotent group $G$ with $\nil(G)$ of even order is non-Eulerian.



\section{Preliminaries}
\label{Section1}

\subsection*{Basics on Graph Theory.}
We recall some definitions, notations, and results concerning elementary graph theory for the reader's convenience and later use. The reader is referred for undefined terms and concepts of graph theory to \cite{Grossman} or \cite{Nora}.\\

An undirected finite graph or simply a graph  $\Gamma = (V, E)$ consists of a non-empty finite set $V$ whose elements are called vertices and a set $E\subseteq \binom{V}{2}$ whose elements are called edges. Here $\binom{V}{2}$ is the set of two-element subsets of $V$. Typically, the letters $n$ and $m$ are used so that $|V| = n$ (the order of $\Gamma$) and $|E| = m$ (the size of $\Gamma$). By this definition, an undirected graph cannot have self-loops since $v\in V$ holds $\{v, v\} = \{v\} \notin \binom{V}{2}$. Further, all the graphs considered in this work are simple; no multiple edges connect two vertices. 

In the following, $\Gamma$ always denotes a graph. If $x$ and $y$ are vertices of $\Gamma$, then it is said that $x$ is adjacent (or related) to $y$ if $\{x, y\}$ is an edge.  In this case, the edge's endpoints are $x$ and $y$. 

A subgraph of $\Gamma$ is a graph $\Gamma'$ whose vertices and edges form subsets of the vertices and edges of $\Gamma$. The neighbourhood of a vertex $w$ in $\Gamma$, denoted by $N_\Gamma(w)$ is defined as the set 
\[
N_\Gamma(w) := \{v \in V : \{v, w\} \in E \}.
\]
The degree of $w$, denoted by $\deg(w)$, is the number of edges incident with $w$. That is, $\deg(w)= |N_\Gamma(w)|$. A vertex of degree zero is called an isolated vertex or isolated point. 

Two graphs  $\Gamma_1=(V_1,E_1)$ and  $\Gamma_2 = (V_2,E_2)$    are said to be isomorphic if there is a bijection $f :  V_1\longrightarrow V_2$  such that 
\[\{u,v\}\in E_1 \ \Longleftrightarrow  \{f(u), f(v)\} \in E_2.\]
	
It means that if an edge joins two vertices in $\Gamma_1$, the corresponding vertices are joined by an edge in $\Gamma_2$.
 
A path in $\Gamma$ is a sequence of distinct vertices $v_1, v_2,\ldots, v_k$ such that $\{v_i, v_{i+1}\}\in E$ for $1\leq i < k$. If $v_1, v_2,\ldots, v_k$ is a path, it is said that it goes from $v_1$ to $v_k$ (or that it connects $v_1$ with $v_k$) and that the vertices $v_1$ and $v_k$ are its endpoints. It is said that $\Gamma$ is connected if there exists a path connecting any two vertices; otherwise, $\Gamma$ is disconnected. The relation defined over the set of vertices $V$ in which two vertices are related if there exists a path between them is an equivalence relation. A connected component of $\Gamma$ is an equivalence class under this relation; or equivalently, a maximal connected subgraph.\\

Some useful notation is the following:
\begin{align*}
\kappa(\Gamma)  & := \text{the number of connected components of $\Gamma$}\\
\mathcal{C}(\Gamma) & :=\{C_1,\ldots, C_{\kappa(\Gamma)}\}, \  \text{the set of all connected components of $\Gamma$} \\
\kappa_j(\Gamma)  & := |C_j|, \ \text{for $j=1,\ldots \kappa(\Gamma)$}.
\end{align*}


A clique in $\Gamma$ is a subgraph in which any two vertices are adjacent, i.e., a complete subgraph of $\Gamma$. A clique with the largest possible size is called a maximum clique.  The size $\omega(\Gamma)$ of a maximum clique in $\Gamma$ is called the clique number of $\Gamma$.

A bipartite graph is a graph in which the vertices can be divided into two disjoint sets such that all edges connect a vertex in one set to a vertex in another. There are no edges between vertices in the disjoint sets.

The complement of $\Gamma$ is a graph $\overline{\Gamma}$ on the same vertices such that two distinct vertices of $\overline{\Gamma}$ are adjacent if and only if they are not adjacent in $\Gamma$.



\subsection*{Basics on Group Theory.} 
The symmetric group on a finite set $\Omega$ is the group whose elements are all permutations of $\Omega$, that is, all bijective functions from $\Omega$ to $\Omega$, and whose group operation is the composition. The symmetric group of degree $n$, denoted by $S_{n}$, is the symmetric group defined on the set $\Omega = [n]:=\{1,...,n\}$. For more details on notation and definitions see \cite{Robinson_1995} or \cite{Huppert}.\\ 

Let $\sigma \in S_{n}$. The expression
\[
\sigma = (a_1 a_2\cdots a_r)
\]
denotes the permutation $\sigma$ that $\sigma(a_1) = a_2,\ldots, \sigma(a_{r-1}) = a_r$, $\sigma(a_r) = a_1$ and $\sigma(a) = a$ for other $a$. This permutation is called a $k$-cycle in $S_{n}$. A 2-cycle $\tau =(a b)$ is a transposition. 
The support of the permutation $\sigma\in S_{n}$, denoted by $\supp(\sigma)$ is defined as follows:
\[
\supp(\sigma) := \{j\in [n]: \sigma(j)\neq j\}.
\]
So all the action takes place in the support. The identity permutation has empty support.
Two permutations $\sigma$ and $\pi$ in $S_{n}$ are disjoint if their supports are disjoint. In  particular, two cycles $\sigma = (a_1 a_2\cdots a_k)$ and $\pi = (b_1 b_2\cdots b_r)$ are disjoint if the underlying sets $\{a_1, a_2,\cdots a_k\}$ and $\{b_1, b_2,\cdots b_r\}$, are disjoint. It is well known that disjoint permutations commute; that is, if $\sigma$ and $\pi$  are disjoint, then $\sigma \pi = \pi\sigma$. Finally, the fix of $\sigma\in S_{n}$ are defined by
\[
\fix(\sigma) := \{x\in [n] : \sigma(x)=x\}.
\]


Let $X$ be a left $G$-set. If  $x\in X$, its orbit is the set  $\orb_G(x)=\{g\cdot x: g \in G\}$, and the stabilizer of $x$, denoted by  $\stab_G(x)= \{g\in G : g\cdot x=x \text{ for all } x\in X\}$.





For any group $G$ a subgroup $\zeta_j(G)$ is defined recursively by setting $\zeta_0(G) = 1$, $\zeta_1(G)=Z(G)$, and $\zeta_{j+1}(G)/\zeta_j(G) = Z(G/\zeta_j(G))$ for $j\geq 1$. Every subgroup $\zeta_j(G)$ is characteristic in $G$. The chain 
\[
1 = \zeta_0(G) \leq \zeta_1(G) \leq \zeta_2(G) \leq \cdots
\]
is called the upper, or ascending, central series of $G$. It is said that $G$ is nilpotent, of  nilpotency class $c$, if $\zeta_{c-1}(G) < \zeta_c(G)=G$ for some integer $c$, depending of $G$. The characteristic subgroup 
\[
Z^\ast(G) := \bigcup_{j=0}^\infty \zeta_j(G)
\]
is called the hypercenter of $G$. 

For example, the hypercenter of the symmetric group of degree $n$, denoted by $S_{n}$, is trivial for $n \neq 2$.

\begin{example}\cite[Example 2.9]{Clement2017}. 
\label{NilpotentizerDihedral}
Let $n \in \mathbb{Z}_{>0}$, with $n = 2^km$, $m \geq 3$ and $\gcd(2, m) = 1$. Then $D_n/Z^*(D_n) \cong D_m$.
\end{example}


It is well known that $G$ is nilpotent if and only if the central ascending series is stationary in $G$.  

From now on, the symbol $G$ denotes a finite group. Some of the fundamental characterizations of a finite nilpotent group are presented below. This list of equivalent statements to the nilpotency of $G$ will be essential in later sections.

 \begin{theorem}\cite[Theorem 5.2.4]{Robinson_1995}.
\label{NilpotentEquivalence}
The following statements about $G$ are equivalent: 
    \begin{enumerate}
        \item $G$ is nilpotent
        \item  Every subgroup of $G$ is subnormal.
        \item Every proper subgroup of $G$ is properly contained in its normalizer.
        \item Every maximal subgroup of $G$ is normal.
        \item  Every Sylow subgroup of $G$ is normal.
        \item $G$ is the direct product of its Sylow subgroups.
        \item Any two elements whose orders are relative primes, commute.
    \end{enumerate}
\end{theorem}

\begin{theorem}\cite[Theorem 2.5]{Clement2017}.
\label{ClementNil}
Let $N\unlhd G$. If $N \leq \zeta_i(G)$ for some $i\in \N$ and $G/N$ is nilpotent, then $G$ is nilpotent.
\end{theorem}

The notation $\pi(G)$ denotes the set of all primes that divide the order of $G$. 
For a prime number $p$, $O_p(G)$ denotes the intersection of all Sylow $p$-subgroups of $G$, and it is well known that $O_p(G)$ is the largest normal $p$-subgroup of $G$. The Fitting subgroup of $G$, denoted by $\mathrm{Fit}(G)$ or $\mathrm{F}(G)$ is defined as follows:
\[
\mathrm{F}(G):= \langle O_{p}(G):  p\in  \pi(G)\rangle. 
\]
It is clear from this definition that $\mathrm{F}(G)$ is the direct product $\prod O_p(G)$ over the prime divisors $p$ of $|G|$ and that $O_p(G)$ is a Sylow $p$subgroup of $\mathrm{F}(G)$. Then, using Theorem \ref{NilpotentEquivalence}, the Fitting subgroup is a nilpotent subgroup of $G$. Further, it is the largest normal nilpotent subgroup of $G$.


\section{The nilpotent graph of a finite group and some examples}\label{Section2}

This section provides a SageMath \cite{SageMath} algorithm to compute the nilpotent graph of a non-nilpotent group, along with some examples of nilpotent graphs that will be useful to illustrate later results. It is proved that the nilpotent graph of a finite group $G$ is never a star and that it is bipartite if and only if $G$ is isomorphic to the symmetric group of degree three. A formula for the degree of a nilpotent graph is offered.\\ 

The letter $\frakN$ denotes the class of all finite nilpotent groups through this paper.
 If $h\in G$,  the nilpotentizer of $h$ in $G$ is the set $\nil_G(h) := \{g\in G:  \langle g, h\rangle \ \text{is nilpotent} \}$. The nilpotentizer of $G$ is the set $\nil(G) :=\bigcap_{h\in G}\nil_G(h)= \{g\in G:  \langle g, h\rangle \ \text{is nilpotent for all} \ h\in G\}$.\\

It is not known exactly when the subset $\nil(G)$ is a subgroup of $G$, but in some cases it is. For instance, if a group $G$ satisfies the maximal condition on its subgroups or $G$ is a finitely generated soluble group, then  $Z^\ast(G) = \nil(G) = R(G)$, where $R(G)$ denotes the soluble radical of $G$, which is the largest soluble normal subgroup of $G$ (see \cite{Abdollahi1} or \cite[Lemma 3.1]{Das2}). In particular, if $G$ is finite, this holds. 




Let $G$ be a non-nilpotent finite group. The nilpotent graph of $G$, denoted by $\Gamma_\frakN(G)$, is an undirected graph whose vertex set is $G- \nil(G)$, and two vertices $g$ and $h$ are adjacent if and only if $\langle g, h\rangle$ is a nilpotent subgroup of $G$.\\


\begin{proposition}
 $\Gamma_\frakN(G)$ is bipartite if and only if $G \cong S_{3}$.
\end{proposition}

\begin{proof} 
If $\Gamma_\frakN(G)$ is bipartite, then it has no $3$-cycles. Then, by \cite[Proposition 3.2]{Das2}, $G \cong S_{3}$.

Reciprocally, if $G \cong S_{3}$, note that the unique edge in $\Gamma_\frakN(G)$ is incident with  $(123)$ and $(132)$, see Example \ref{Sym3}. Then, the assertion holds.
\end{proof}

The following code in $\textit{SageMath}$ \cite{SageMath} is used to generate the nilpotent graph of a finite group $G$:

 \begin{algorithm}[H]
\KwData{$G$ = Group() \# Any group supported by SageMath.}
\KwResult{The nilpotent graph of $G$}
\begin{verbatim}
    G = Group() # Any group supported by SageMath can go here.
    nil_graph = Graph()
    PAIRS = []
    elements = G.list()
    nilpotentizer = []
    for element1 in elements:
        is_in_nilpotentizer = True
        temp = []
        for element2 in elements:
            if element1 != element2:
                H = G.subgroup([element1, element2])
                if H.is_nilpotent():
                    temp.append((element1, element2))
    
                else:
                    is_in_nilpotentizer = False
    
        if not is_in_nilpotentizer:
            for t in temp:
                PAIRS.append(t)
        elif is_in_nilpotentizer:
            nilpotentizer.append(element1)
    for p in PAIRS:
        if p[0] not in nilpotentizer and p[1] not in nilpotentizer:
            nil_graph.add_edge(p)
    
    P = nil_graph.plot(vertex_labels=False)
    P.show(figsize=(8, 8))
\end{verbatim}

\caption{An algorithm to generate the nilpotent graph of a finite group $G$ after removing its isolated points}
\end{algorithm}

\begin{example}[The nilpotent graph of the symmetric group of degree three]    \label{Sym3}
$\Gamma_\frakN(G)$ has four connected components, three of which are isolated points.

\begin{center}
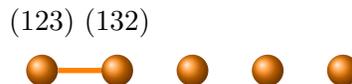
\begin{figure}[H]
\begin{tikzpicture}
\SetVertexMath
\SetGraphUnit{1}
\GraphInit[vstyle=Art]
\renewcommand*{\EdgeLineWidth}{2pt}

\Vertex[x=0, y=1, L=i_1]{i1}
\Vertex[x=1, y=1, L=i_2]{i2}

\Vertex[x=2, y=1, L=i_3]{i3}
\Vertex[x=3, y=1, L=i_4]{i4}
\Vertex[x=4, y=1, L=i_4]{i5}
  
\Edges(i1, i2)
 
\node[above] at (0,1.3) {$(123)$};
  \node[above] at (1,1.3) {$(132)$}; 
\end{tikzpicture}
\caption{The nilpotent graph of $S_{3}$}
\end{figure}
\end{center}
\end{example}

\begin{example}[The nilpotent graph of the symmetric group of degree four]\label{nils4}
$\Gamma_\frakN(S_{4})$ has five connected components.
\begin{center}
\begin{figure}[H]
\begin{tikzpicture}
\SetVertexMath
\SetGraphUnit{2}
\GraphInit[vstyle=Art]
\renewcommand*{\EdgeLineWidth}{1.2pt}
\Vertex[x=4, y=0]{1} 
\Vertex[x=3.65, y=1.63]{2} 
\Vertex[x=2.68, y=2.97]{3} 
\Vertex[x=1.24, y=3.8]{4} 
\Vertex[x=-0.42, y=3.98]{5} 
\Vertex[x=-2, y=3.46]{6} 
\Vertex[x=-3.24, y=2.35]{7} 
\Vertex[x=-3.91, y=0.83]{8} 
\Vertex[x=-3.91, y=-0.83]{9} 
\Vertex[x=-3.24, y=-2.35]{10} 
\Vertex[x=-2, y=-3.46]{11} 
 \Vertex[x=-0.42, y=-3.98]{12} 
\Vertex[x=1.24, y=-3.8]{13} 
\Vertex[x=2.68, y=-2.97]{14} 
\Vertex[x=3.65, y=-1.63]{15} 

\Vertex[x=6, y=-3]{16}
\Vertex[x=6, y=-2]{17}

\Vertex[x=6, y=1]{18}
\Vertex[x=6, y=2]{19}

\Vertex[x=7, y=-3]{20}
\Vertex[x=7, y=-2]{21}

\Vertex[x=7, y=1]{22}
\Vertex[x=7, y=2]{23}

\Edges(4,2,3) 
\Edges(3,15,2,14,5) 
\Edges(5,12,2,13,5) 
\Edges(5,11,2,10,5) 
\Edges(5,8,2,9,5) 
\Edges(5,7,2,6,10,7) 

\Edges(3,1,2)
\Edges(4,6,1,5,2)
\Edges(4,15,1,4,3)

\Edges(11,12,6,3,5,4)
\Edges(11,14,6,15,5,6,9)
\Edges(11,6,13,11)
\Edges(8,9,7,6,8,7)
\Edges(8,10,9)
\Edges(14,12,13,14)
\Edges(16,17)
\Edges(18,19)
\Edges(20,21)
\Edges(22,23)

\node[right] at (4.2,0) {$(12)$};
\node[right] at (3.8,1.63) {$(12)(34)$};
\node[right] at (2.8,2.97) {$(1423)$};
\node[right] at (1.4,3.9) {$(34)$};
\node[above] at (-0.42,4.1) {$(13)(24)$};
\node[above] at (-2.2,3.6) {$(14)(23)$};
\node[left] at (-3.4,2.35) {$(1342)$};
\node[left] at (-4.1, 0.83) {$(23)$};
\node[left] at (-4.1,-0.83) {$(14)$};
\node[left] at (-3.4,-2.5) {$(1243)$};
\node[below] at (-2.2,-3.6) {$(1432)$};
\node[below] at (-0.42,-4.1) {$(13)$};
\node[below] at (1.4,-4) {$(24)$};
\node[right] at (2.8, -3) {$(1234)$};
\node[right] at (3.8,-1.63) {$(1324)$};

\node[below] at (6,-3.2) {$(134)$};
\node[above] at (6,-1.8) {$(143)$};

\node[below] at (6,0.8) {$(123)$};
\node[above] at (6,2.2) {$(132)$};

\node[below] at (7,-3.2) {$(142)$};
\node[above] at (7,-1.8) {$(124)$};

\node[below] at (7,0.8) {$(234)$};
\node[above] at (7,2.2) {$(243)$};
\end{tikzpicture}
	\caption{The nilpotent graph of $S_{4}$.}
\end{figure}
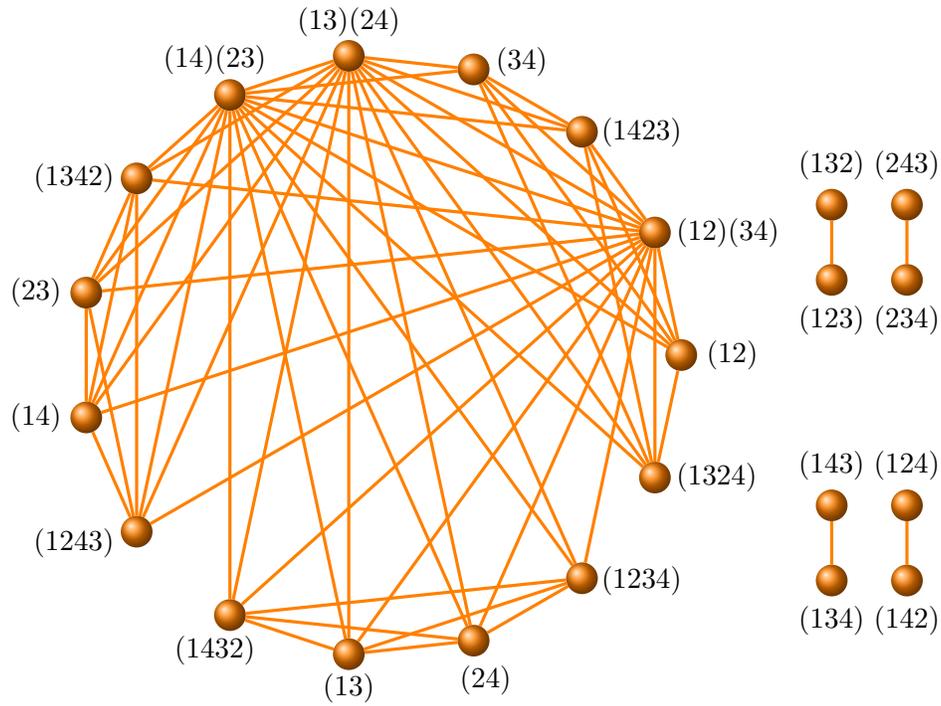
\end{center}
\end{example}



 





\begin{example}\label{nilD5}
The nilpotent graph of $D_5$ has $6$ connected components. 
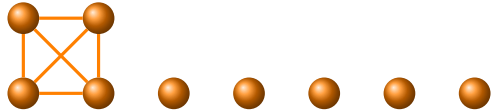
\begin{figure}[H]
\begin{center}
\begin{tikzpicture}
\SetVertexMath
\SetGraphUnit{2}
\GraphInit[vstyle=Art]
\renewcommand*{\EdgeLineWidth}{1.2pt}

\Vertex[x=3, y=1]{i1}
\Vertex[x=3, y=0]{i2}
\Vertex[x=4, y=1]{i3}
\Vertex[x=4, y=0]{i4}

\Vertex[x=5, y=0]{j1}
\Vertex[x=6, y=0]{j2}
\Vertex[x=7, y=0]{j3}
\Vertex[x=8, y=0]{j4}
\Vertex[x=9, y=0]{j5}
 
\Edges(i1, i2, i3, i4, i1, i3)
\Edges(i2, i4)
\end{tikzpicture}
\end{center}
 \caption{The nilpotent graph of $D_{5}$}
 \label{D5}
\end{figure}

\end{example}

\begin{example}\label{nilD12}
The nilpotent graph of $D_{12}$ has $4$ connected components.
\begin{center}
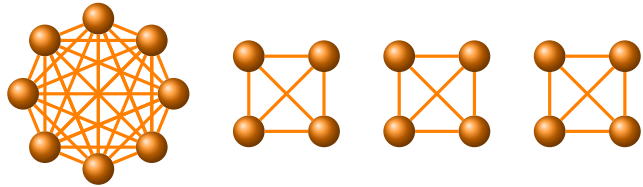
\begin{figure}[H]
\begin{tikzpicture}
\SetVertexMath
\SetGraphUnit{2}
\GraphInit[vstyle=Art]
\renewcommand*{\EdgeLineWidth}{1.2pt}

\Vertex[x=1, y=0, L=k1]{k1}
 \Vertex[x=0.71, y=0.71, L=k2]{k2}
 \Vertex[x=0, y=1, L=k3]{k3}
 \Vertex[x=-0.71, y=0.71, L=k4]{k4}
 \Vertex[x=-1, y=0, L=k5]{k5}
 \Vertex[x=-0.71, y=-0.71, L=k6]{k6}
 \Vertex[x=0, y=-1, L=k7]{k7}
 \Vertex[x=0.71, y=-0.71, L=k8]{k8}

\Vertex[x=2, y=0.5, L=i1]{i1}
\Vertex[x=2, y=-0.5, L=i2]{i2}
\Vertex[x=3, y=0.5, L=i3]{i3}
\Vertex[x=3, y=-0.5, L=i4]{i4}

\Vertex[x=4, y=0.5, L=i5]{i5}
\Vertex[x=4, y=-0.5, L=i6]{i6}
\Vertex[x=5, y=0.5, L=i7]{i7}
\Vertex[x=5, y=-0.5, L=i8]{i8}
 
\Vertex[x=6, y=0.5, L=j1]{j1}
\Vertex[x=6, y=-0.5, L=j2]{j2}
\Vertex[x=7, y=0.5, L=j3]{j3}
\Vertex[x=7, y=-0.5, L=j4]{j4}

\Edges(i1, i2, i3, i4, i1, i3)
\Edges(i2, i4)

\Edges(i5, i6, i7, i8, i5, i7)
\Edges(i6, i8)

\Edges(j1, j2, j3, j4, j1, j3)
\Edges(j2, j4)
 
 \Edges(k2,k7,k1,k2,k3,k4,k5,k6,k7,k8,k1,k4,k2,k5,k3)
 \Edges(k7,k5,k1,k3,k6,k4,k8,k5)
 \Edges(k1,k6,k2,k8,k3,k7,k4)
 \Edges(k1,k6,k2,k8,k3,k7,k4)
 \Edges(k6,k8)
\end{tikzpicture}
 \caption{The nilpotent graph of $D_{12}$}
 \label{D12}
\end{figure}
\end{center}
\end{example}

\begin{example}[The nilpotent graph of the projective special linear group of degree two over $\F_3$]
$\Gamma_\frakN(G)$ has five connected components.

\begin{center}
\begin{figure}[H]
\begin{tikzpicture}
\SetVertexMath
\SetGraphUnit{1}
\GraphInit[vstyle=Art]
\renewcommand*{\EdgeLineWidth}{2pt}

\Vertex[x=2, y=0, L=i_1]{i1}
\Vertex[x=2, y=1, L=i_2]{i2}

\Vertex[x=3, y=0, L=i_3]{i3}
\Vertex[x=3, y=1, L=i_4]{i4}
 
\Vertex[x=4, y=0, L=j_1]{j1}
\Vertex[x=4, y=1, L=j_2]{j2}

\Vertex[x=5, y=1, L=j_3]{j3}
\Vertex[x=5, y=0, L=j_4]{j4}

\Vertex[x=6, y=0, L=j_3]{k1}
\Vertex[x=6, y=1, L=j_4]{k2}
\Vertex[x=7, y=0.5, L=j_4]{k3}

\Edges(i1, i2)
\Edges(i3, i4)
\Edges(j1, j2)
\Edges(j3, j4)
\Edges(k1, k2,k3,k1)
\end{tikzpicture}
\caption{The nilpotent graph of $\psl(2,3)$}
     \label{PSL(2,3)}
\end{figure}
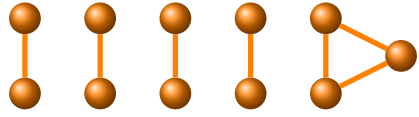
\end{center}
\end{example}

\begin{remark}
Contrary to what has been proven for the soluble graph of a non-soluble finite group (see for example \cite{Lucchini3}, \cite{Bhowal}), the nilpotent graph is generally not connected.  
\end{remark}

Lemmas \ref{grado} and \ref{star} follow similar arguments to the ones given in \cite[Lemma 2.1 and Proposition 2.2]{Bhowal}.

\begin{lemma}\label{grado}
For every $x\in G - \nil(G)$
\[
  \deg(x) = |\nil_G(x)| - |\nil(G)| - 1.  
\]
\end{lemma}

\begin{proof}
Since $\nil(G)\subseteq \nil_{G}(x)$ for all $x\in G$,  then $\nil_{G}(x)- \nil(G) \subseteq G - \nil(G)$. Thus if $x\in G - \nil(G)$, the number of vertices $h$ distinct form $x$ such that $\langle x,h\rangle$  is nilpotent is $|\nil_G(x)| - |\nil(G)|$ and so  $ \deg(x) = |\nil_G(x)| - |\nil(G)| - 1$. 
\end{proof}

\begin{lemma}\label{star}
The graph $\Gamma_\frakN(G)$ is not a star.
\end{lemma}

\begin{proof}
Suppose that $\Gamma_\frakN(G)$ is a star. Then there exists $x\in G- \nil(G)$ such that $\deg(x) = |G- \nil(G)|-1=  |G|-|\nil(G)|-1$. Thus, by Lemma \ref{grado},  $|\nil_G(x)| = |G|$ so that $x\in \nil(G)$, which is a contradiction.
\end{proof}

%


\section{Strongly self-centralizing subgroups and nilpotent graphs}
\label{Section3}

This section provides several lower bounds on the clique number and the number of connected components of the nilpotent graph of a finite group. In addition, it is proved that the nilpotent graph of $S_{n}$ (for $n\geq 5$) is disconnected if and only if $n$ or $n-1$ is a prime number.\\

A subgroup $U$ of a finite group $G$ is said to be strongly self-centralizing if $C_G(x) = U$ for all $1\neq x\in U$.\\

Form now on, $G$ denotes a finite non-nilpotent group in the remainder of this work, unless stated otherwise.

\begin{lemma} \label{nilclique}
 Let $U$ be a nilpotent subgroup $G$ such that $U\nsubseteq \nil(G)$. Then  $U- (U\cap \nil(G))$ is the set of vertices of a clique of $\Gamma_\frakN(G)$. Moreover if $N$ is a normal nilpotent subgroup of $G$, $|N-(N\cap \nil(G))|\leq \omega(\Gamma_\frakN(G))$. In particular,  $|\mathrm{F}(G)-(\mathrm{F}(G)\cap \nil(G))|\leq \omega(\Gamma_\frakN(G))$. 
\end{lemma}  

\begin{proof}
Since any subgroup of a nilpotent group is nilpotent, $\langle x,y\rangle$ is nilpotent for all $x,y\in U-(U\cap \nil(G))$. Thus the subgraph of $\Gamma_\frakN(G)$  having $U- (U\cap \nil(G))$ as its set of vertices is a clique. Moreover, if $N$ is a normal nilpotent subgroup of $G$, then   
\begin{eqnarray*}
|N-(N\cap \nil(G))|  &\leq&\max\{|U- (U\cap \nil(G))|: U\trianglelefteq G \wedge U \in \frakN \}\\
                 &\leq& \max\{|U- (U\cap \nil(G))|: U\leq G \wedge U \in \frakN \}\\
                 &\leq& \omega(\Gamma_\frakN(G)).
\end{eqnarray*}
The rest follows from the fact that the fitting subgroup $\mathrm{F}(G)$ is the largest normal nilpotent subgroup of $G$.
\end{proof}

\begin{theorem}
\label{ProperCentralizer2}
 If there exists a strongly self-centralizing subgroup $U$ of $G$, then:

 \begin{enumerate}
     \item $\nil_G(x) = U$ for all $1\neq x\in U$.\\
     
     \item $\Gamma_\frakN(G)$ is a disconnected graph.\\

     \item  $|\mathrm{F}(G)|-1\leq \omega(\Gamma_\frakN(G))$.
 \end{enumerate}
\end{theorem}

\begin{proof}

\begin{enumerate}
    \item Suppose a strongly self-centralizing subgroup $U$ of $G$ exists. Note that $Z(G) \leq C_G(x) = U$, for all $1\neq x\in U$. Since $U$ is abelian, then there is $y \in G- U$ such that $U\cap C_G(y) = 1$; otherwise it would exist $1 \neq z \in U\cap C_G(y)$ and therefore $y \in C_G(z) = U$, which is a contradiction. This implies that $Z(G) \leq U \cap C_G(y) = 1$, and so $Z^\ast(G) = 1$.\\
    
    Let $1\neq x\in U$. If $g \in \nil_G(x)$, then there exists $1\neq z \in Z(\langle x, g \rangle)$, and so $z \in C_G(x) = U$, implying that  $g \in C_G(z) = U$. Thus $\nil_G(x) \subseteq U$ and $U=C_{G}(x)\subseteq nil_{G}(x)$, i.e.,  $\nil_G(x) = U$.\\
    
    \item By part $1$, for every $h \in G - U \subseteq G - Z^\ast(G)=G- 1$ and $x\in U$, $\langle x, h \rangle$ is not nilpotent, i.e., there is no edge from $h$ to $x$, so $\Gamma_\frakN(G)$ must be a disconnected graph.\\

    \item Since $G$ is finite, $\nil(G)=Z^\ast(G)$. Thus, by the proof of part $1$, $nil(G)=Z^\ast(G) = 1$. Therefore $|\mathrm{F}(G)|-1\leq \omega(\Gamma_\frakN(G))$ (by Lemma \ref{nilclique}).
\end{enumerate}
 
\end{proof}



\begin{theorem}\label{str}
Let $\mathcal{U}\neq \emptyset $ be the set of the strongly self-centralizing subgroups of $G$. If $U\in \mathcal{U}$, then $U-1$ is the set of vertices of a connected component of $\Gamma_\frakN(G)$ and 

\[\kappa(\Gamma_\frakN(G))\geq \begin{cases} |\mathcal{U}| \mbox{ if } \bigcup_{U\in \mathcal{U}}U= G\\
                                             |\mathcal{U}|+1 \mbox{ if } \bigcup_{U\in \mathcal{U}}U\neq G.  
    
\end{cases}\]

\end{theorem} 

\begin{proof}
Suppose $\mathcal{U}\neq \emptyset$. Then $\nil(G)=Z^{*}(G)=1$ (by the proof of Theorem \ref{ProperCentralizer2}, part $1$) and so the set of vertices of $\Gamma_\frakN(G)$ is $G-1$. Let $U\in \mathcal{U}$, then  $\nil_G(x) = U$ for all $1\neq x\in U$ (by Theorem \ref{ProperCentralizer2}, part $1$), that is, any vertex in $G-1$ that is adjacent to a vertex in $U-1$ must be in $U-1$. Note that if $C$ is a connected component of $\Gamma_\frakN(G)$ containing a vertex $x \in U-1$ and $y\in C$, then there exists a path from $x$ to $y$, so the vertex of that path that is connected to $x$ must be in $U-1$. Any other vertex connected to that one must be in $U-1$ as well. By applying this reasoning consecutively, one gets that $y\in U-1$. Thus, since  $U-1$ is the set of vertices of a connected subgraph of $\Gamma_\frakN(G)$ (by Lemma \ref{nilclique}), such subgraph must be a connected component of $\Gamma_\frakN(G)$. 

 If $\bigcup_{U\in \mathcal{U}}U\subsetneq G$, then any element $ h\in G- \bigcup_{U\in \mathcal{U}}U$ is a vertex $\Gamma_\frakN(G)$ that is not related to an element of $\bigcup_{U\in \mathcal{U}}(U-1)$ so that $h$ belongs to a connected component whose set of vertices has trivial intersection with $\bigcup_{U\in \mathcal{U}}(U-1)$, implying that $|\mathcal{U}|+1\leq \kappa(\Gamma_\frakN(G))$. Therefore

\[\kappa(\Gamma_\frakN(G))\geq \begin{cases} |\mathcal{U}| \mbox{ if } \bigcup_{U\in \mathcal{U}}U= G\\
                                             |\mathcal{U}|+1 \mbox{ if } \bigcup_{U\in \mathcal{U}}U\neq G  
    
\end{cases}\]
\end{proof}

\begin{example}\label{ncomp-nilS4}
Using SageMath \cite{SageMath}, one can see that the strongly self-centralizing subgroups of $S_{4}$ are
\[\langle (132) \rangle, \langle (142) \rangle, \langle (143) \rangle, \langle (243) \rangle.\]
Thus, if $\mathcal{U}$ is defined as in Theorem \ref{str}, $|\mathcal{U}|+1=4+1=5\leq \kappa(\Gamma_\frakN(G))=5$ where the last equality is by Example \ref{nils4}. \\

Similarly, one can see that the strongly self-centralizing subgroups of $D_{5}$ are
\[\langle (25)(34) \rangle,  \langle (12)(35) \rangle,  \langle (13)(45) \rangle,\langle (14)(23) \rangle,\langle (15)(24) \rangle, \langle (12345) \rangle.\]
Thus, if $\mathcal{U}$ is defined as in Theorem \ref{str}, $|\mathcal{U}|=6\leq \kappa(\Gamma_\frakN(G))=6$  where the last equality is by Example \ref{nilD5}.
\end{example}

\begin{example}
 From the shape of $\Gamma_\frakN(S_{4})$ and $\Gamma_\frakN(D_{5})$, that were computed in Examples \ref{nils4} and \ref{nilD5}, respectively, it can be seen that $\omega (\Gamma_\frakN(S_{4})\newline) =5$ and $\omega (\Gamma_\frakN(D_{5}))=4$. On the other hand, $ F(S_{4})=\{ 1, (12)(34), (14\newline )(23), (13)(24)\}$  and $ F(D_{5})=\{ 1, (12345), (13524), (14253), (15432)\}$. Thus, since $S_{4}$ and $D_{5}$ contain strongly self-centralizing subgroups (these were computed in Example \ref{ncomp-nilS4}), \[|F(S_{4})|-1=3\leq 5=\omega (\Gamma_\frakN(S_{4}))\]  and \[|F(D_{5})|-1=4\leq 4=\omega (\Gamma_\frakN(D_{5})).\] These inequalities illustrate Theorem \ref{ProperCentralizer2} (part $3$).
\end{example}




\begin{theorem}
Let $n\in \mathbb{Z}_{>0}$.

\begin{enumerate}
    \item If $n$ or $n-1$ is a prime number, then $S_{n}$ has self-centralizing cyclic subgroups. In particular, $\Gamma_\frakN(S_{n})$ is a disconnected graph.\\

    \item If $n \geq 5$ and $\Gamma_\frakN(S_{n})$ is disconnected, then $n$ or $n-1$ is a prime number.
\end{enumerate}

\end{theorem}

\begin{proof}
\begin{enumerate}
    \item 
Suppose that $n$ or $n-1$ is a prime number. Let $G:=S_{n}$ and consider the action of $G$ on itself by conjugation. It is known that the conjugacy classes of any $r$-cycle in $G$ are determined by cycle type. Thus for $1 \leq r \leq n$, the number of $r$-cycles in $G$ is $\frac{n!}{r(n-r)!}$. If $\sigma$ is an $r$-cycle in $G$, then
\[
|G:C_G(\sigma)| = |\orb_G(\sigma)| = \frac{n!}{r(n-r)!}.
\]
Let  $\alpha$ and $\beta$ be an $n$-cycle and an $(n-1)$-cycle in $G$ respectively. Then  
\[
|G : C_G(\alpha)| = (n-1)! \ \ \text{and} \ \ |G : C_G(\beta)| = (n-2)!n. 
\]
Therefore $|C_G(\alpha)| = n$ y $|C_G(\beta)| = n-1$. It follows that $C_G(\alpha) = \langle \alpha \rangle$ and $C_G(\beta) = \langle \beta \rangle$. If $n$ is a prime number then $C_G(\alpha^k) = \langle \alpha \rangle$ for $1 \leq k < n-1$. Analogously, if $n-1$ is a prime number then $C_G(\beta^k) = \langle \beta \rangle$, for $1 \leq k < n$. Thus, by Theorem \ref{ProperCentralizer2} (part $2$), $\Gamma_\frakN(S_{n})$ is a disconnected graph.\\

\item Suppose that $\Gamma_\frakN(S_{n})$ is a disconnected graph, and let $\sigma$ be a $(n-1)$-cycle. If $n-1$ is an even number, then there exists an involution $\iota \in S_{n}$ such that $\sigma^m = \iota = (s_1 s_2)(s_3 s_4) \cdots(s_{l-1} s_l)$, for some integer number $m$. Let $\tau_1 := (s_1 s_2)$.

Let $\alpha \in S_{n}$ and its decomposition in disjoints cycles given by
\[
\alpha = (a_1\ldots a_{k_1})(a_{k_1+1}\ldots a_{k_2})\cdots (a_{k_{l-1}+1} \ldots a_{k_l}),   
\]
with $k_i \leq n-2$, $i \in \{1,\ldots, l\}$. Let $\theta := (a_1\ldots a_{k_1})$. Consider the following cases:\\

\noindent \textbf{Case 1:} $\supp(\theta) \cap \supp(\tau_1) = \emptyset$. Then $\tau_1\theta = \theta\tau_1$, and $\tau_1 \in \nil_G(\theta)$.\\

\noindent \textbf{Case 2:} $\{s_1, s_2\}\subseteq \supp(\theta)$. Since $n \geq 5$, there exists another transposition $\tau_2 =(s_1' s_2')$ such that $\theta(s_1')=s_1'$ and $\theta(s_2')=s_2'$. Therefore, $\tau_2\theta = \theta\tau_2$ y $\tau_1\tau_2 = \tau_2\tau_1$. Then $\theta\in \nil_G(\tau_2)$ and $\tau_2\in \nil_G(\tau_1)$.\\

\noindent \textbf{Case 3:} $s_1\in \supp(\theta)$, and $\theta(s_2)=s_2$. Let $\tau_3= (s_2 s)$, with $s\in \fix(\theta)$, which exists, since $n \geq 5$ and $\theta$ has length at most $n-2$. Then $\theta$ and $\tau_3$ are disjoints. Thus, there exists $\tau_4 = (b_1 b_2)$, where $b_1$ and $b_2$ are elements that $\tau_3$ and $\tau_1$ stabilize, since $n \geq 5$. Therefore, $\tau_3\theta = \theta\tau_3$, $\tau_4\tau_1 = \tau_1\tau_4$, and $\tau_3\tau_4 = \tau_4\tau_3$. 
It follows that $\theta \in \nil_G(\tau_3)$, $\tau_3\in \nil_G(\tau_4)$ y $\tau_4 \in \nil_G(\tau_1)$.\\

\noindent \textbf{Case 4:} $s_2\in \supp(\theta)$, and $\theta(s_1)=s_1$. Similar to the above case.

Since $\theta \in \nil_G(\alpha)$, $\tau_1 \in \nil_G(\sigma^m)$ and $\sigma^m \in \nil_G(\sigma)$, there is always a path between any $(n-1)$-cycles and permutations that are the product of disjoint cycles of length less than $n-1$, other than 1. Therefore, $\Gamma_\frakN(S_{n})$ has a connected component, say $\mathcal{C}$, containing any permutation which is a product of disjoint cycles of length less than or equal to $n-1$. Since $\Gamma_\frakN(S_{n})$ is disconnected, then there must exist at least one other connected component, which contains only $n$-cycles. 

Suppose $n$ is not a prime number. Any cyclic subgroup of $S_{n}$ generated by an $n$-cycle contains a different from 1 permutation, which is a product of disjoint cycles of length less than or equal to $n-1$ so that any $n$-cycle would be adjacent to vertices which are in the connected component $\mathcal{C}$. That is, any $n$-cycle is in $\mathcal{C}$, so all permutations other than the identity are in $\mathcal{C}$, which contradicts that $\Gamma_\frakN(S_{n})$ is disconnected. Then, $n$ must be a prime number. 

On the other hand, if $n-1$ is odd, then $n$ is even, so analogously, $n-1$ is prime.
\end{enumerate}

\end{proof}

\section{Directed product of groups and connectedness}
\label{Section4}

 In this section, necessary and sufficient conditions to guarantee the connectedness of the nilpotent graph (see Proposition \ref{ConnectedProduct} and Theorem \ref{KappaProductNilpotent}); and for two nilpotent graphs to be isomorphic (see Theorem \ref{GraphIsomorphism} and Corollary \ref{OrderGroups}), are provided. It is also proved that the nilpotent graph of any finite non-nilpotent group is never self-complementary (Theorem \ref{not-seft-comp}), among other results.


\begin{lemma}
\label{ComponentsNilAdjacency}
Let $G$ and $H$ be finite groups. Then $\langle (g_1, h_1), (g_2, h_2) \rangle$ is a nilpotent subgroup of $G\times H$ if and only if $\langle g_1, g_2 \rangle$ and $\langle h_1, h_2 \rangle$ are nilpotent subgroups of $G$ and $H$ respectively.
\end{lemma}

\begin{proof}
Suppose $U := \langle (g_1, h_1), (g_2, h_2) \rangle$ is a nilpotent subgroup of $G\times H$, and $\Tilde{H} := 1\times H$. Then 
$U/(U\cap \Tilde{H}) \cong \langle g_1, g_2 \rangle$, and so $\langle g_1, g_2 \rangle$ is nilpotent. Similarly $\langle h_1, h_2 \rangle$ is nilpotent.

Reciprocally, suppose $\langle g_1, g_2 \rangle$ and $\langle h_1, h_2 \rangle$ are nilpotent groups. Then $\langle g_1, g_2 \rangle \times \langle h_1, h_2 \rangle$ is nilpotent. Since $\langle (g_1, h_1), (g_2, h_2) \rangle \leq \langle g_1, g_2 \rangle \times \langle h_1, h_2 \rangle$, then  $\langle (g_1, h_1), (g_2, h_2) \rangle$ is a nilpotent group.
\end{proof}

\begin{corollary}
\label{NilProduct}
Let $G$ and $H$ be groups. Then $\nil(G\times H) = \nil(G)\times \nil(H)$.
\end{corollary} 

\begin{proof}
\begin{align*}
(g, h) \in \nil(G) \times \nil(H) & \Leftrightarrow \langle g, x \rangle, \langle h, y \rangle \in \frakN \ \ \forall x \in G \ \ \forall y \in H\\
& \Leftrightarrow \langle (g, h), (x, y) \rangle \ \in\frakN \ \  \forall  (x, y) \in G \times H\\
& \Leftrightarrow (g, h) \in \nil(G \times H).
\end{align*}  
\end{proof}

\begin{proposition}
\label{ConnectedProduct}
Let $G$ and $H$ be non-nilpotent groups. Then $\Gamma_{\frakN}(G\times H)$ is connected.
\end{proposition}

\begin{proof}
Let $(g_1, h_1), \, (g_2, h_2) \in G \times H - Z^*(G \times H)$. Then $g_1 \notin Z^*(G)$ or $h_1 \notin Z^*(H)$ and $g_2 \notin Z^*(G)$ or $h_2 \notin Z^*(H)$. Consider the following cases:

\noindent \textbf{Case 1.} $g_1 \notin Z^*(G)$ and $h_2 \notin Z^*(H)$. In this case the path $(g_1, h_1)$, $(g_1, 1)$, $(1, h_2)$, $(g_2, h_2)$ connecting the vertices $(g_1, h_1), (g_2, h_2)$. 

\noindent \textbf{Case 2.} $g_1 \in Z^*(G)$ and $g_2 \in Z^*(G)$. Since $G$ is non-nilpotent, there exists $g \in G - Z^*(G)$. Then $(g_1, h_1)$, $(g, 1)$, $(g_2, h_2)$ is a path connecting   $(g_1, h_1)$ and $(g_2, h_2)$. 

The remaining cases are analogous to one of the above.
\end{proof}

\begin{example}
An illustration of the above result is as follows:
\begin{figure}[H]
\centering
\includegraphics[width=0.5\textwidth]{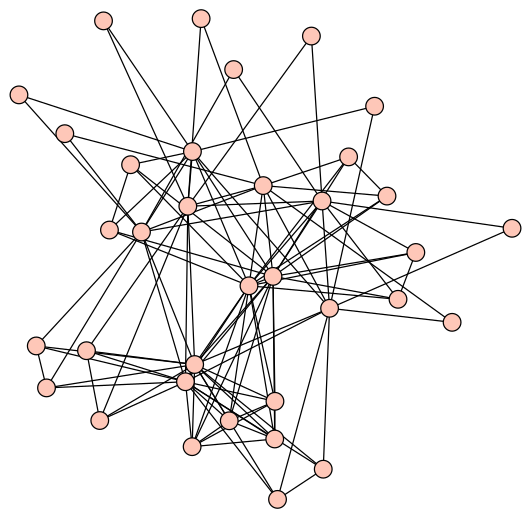}
\caption{Nilpotent graph of $S_{3} \times S_{3}$}
\end{figure}
\end{example}

Let $\kappa_j(\Gamma_{\frakN}(G))$ denote the number of vertices of the $j$-th connected component of $\Gamma_\frakN(G)$. To simplify the notation, from now on, it will be used $\mathcal{C}(G)$, $\kappa(G)$, and $k_j(G)$, to denote $\mathcal{C}(\Gamma_\frakN(G))$, $\kappa(\Gamma_{\frakN}(G))$ and $\kappa_j(\Gamma_{\frakN}(G))$, respectively. 


\begin{theorem}
\label{KappaProductNilpotent}
For every finite nilpotent group $H$, it holds $ \kappa(G \times H) = \kappa(G)$. In particular, $\Gamma_{\frakN}(G)$ is connected if and only if $\Gamma_{\frakN}(G\times H)$ is connected for every finite nilpotent group $H$.
\end{theorem}

\begin{proof}
Let $H$ be any finite nilpotent group. Then $Z^\ast(H) = \nil(H) = H$. So, by Corollary \ref{NilProduct}, $\nil(G \times H) = \nil(G) \times H$. Thus, by Lemma \ref{ComponentsNilAdjacency},  
$(g,h), (x,y)$ are adjacent in $\Gamma_{\frakN}(G\times H)$ if and only if $g$ and $x$ are adjacent in $\Gamma_{\frakN}(G)$. Then, there is no path from $g$ to $x$ if and only if there is no path from $(g,h)$ to $(x,y)$, for all $h, y \in H$. This implies that $\kappa(G \times H) = \kappa(G)$.
\end{proof}

\begin{proposition}
\label{NilDivisor}
Let $H$ be a finite nilpotent group. Then $|H|$ divides $k_i(G\times H)$ for all $i\in \{1,\ldots, \kappa(G\times H)\}$. 
\end{proposition}

\begin{proof}
Let $g$ be a vertex in the $i$-th connected component of $\Gamma_{\frakN}(G)$. Then, by Lemma 
\ref{ComponentsNilAdjacency}, for all $h, k \in H$, $(g, h)$ and $(g, k)$ are in the same connected component of $\Gamma_{\frakN}(G\times H)$. Then, for every $g$ in $i$-th component of $\Gamma_{\frakN}(G)$ there are $|H|$ elements of the form $(g, h)$, with $h \in H$, in the $i$-th connected component of $\Gamma_{\frakN}(G\times H)$.
\end{proof}

\begin{lemma}
\label{QuotientNilAdjacency}
Let $x, y \in G$. Then $\langle x, y\rangle$ is a nilpotent subgroup of $G$ if and only if $\langle xZ^*(G), yZ^*(G) \rangle$ is a nilpotent subgroup of $G/Z^*(G)$.
\end{lemma}

\begin{proof} 
Suppose that $\langle x, y\rangle$ is a nilpotent subgroup of $G$, and let $\pi: G\longrightarrow G/Z^*(G)$ be the canonical epimorphism. Note that $\langle xZ^*(G), yZ^*(G) \rangle = \pi(\langle x, y\rangle)$ and then it is a nilpotent subgroup of $G/Z^*(G)$. 

Reciprocally, suppose that $\langle xZ^*(G), yZ^*(G) \rangle$ is a nilpotent subgroup of $G/Z^*(G)$. First note that $Z^*(G) \cap \langle x, y\rangle \leq Z^*(\langle x, y \rangle)$. In fact, if $z \in Z^*(G) \cap \langle x, y \rangle$, then $z \in \langle x, y \rangle$ and $\langle z, g \rangle$ is nilpotent for all $g \in \langle x, y \rangle$. That is, $z \in Z^*(\langle x, y \rangle)$. Now, it is clear that $Z^*(G) \cap \langle x, y\rangle \, \unlhd \, \langle x, y\rangle$, and consequently $Z^*(G) \cap \langle x, y\rangle \, \unlhd \, Z^*(\langle x, y\rangle)$. Then
\[
\langle xZ^*(G), yZ^*(G) \rangle =  \pi(\langle x, y\rangle)
\cong \langle x, y \rangle/(\langle x, y \rangle \cap Z^*(G)).
\]
Since $\langle xZ^*(G), yZ^*(G) \rangle$ is nilpotent, using Theorem \ref{ClementNil} follows that 
$\langle x, y \rangle$ is nilpotent. 
\end{proof}

\begin{theorem}
\label{GraphIsomorphism}
$\Gamma_{\frakN}(G) \cong \Gamma_{\frakN}(G/Z^*(G)\times Z^*(G))$. In particular $\kappa(G) = \kappa(G/Z^*(G))$.
\end{theorem}

\begin{proof}
Let $T= \{g_1,\ldots, g_m\}$ be a left transversal of $Z^*(G)$ in $G$. Let $\widetilde{G} := G/Z^*(G) \times Z^*(G)$, and define $f: \widetilde{G} \longrightarrow  G$ by $f(gZ^*(G), z) = g_iz$,
where $g_i$ satisfies $gZ^*(G) = g_i Z^*(G)$. It is straightforward to show that $f$ is a well-defined function, and also bijective. 




Now, let $(gZ^*(G), z), (hZ^*(G), w) \in \widetilde{G}$ arbitrary. Then, by Lemma \ref{ComponentsNilAdjacency},  $\langle(gZ^*(G), z), \, (hZ^*(G), w)\rangle$ is a nilpotent subgroup of $\widetilde{G}$ if and only if $\langle gZ^*(G), \, hZ^*(G)\rangle$ is a nilpotent subgroup of $G/Z^*(G)$. However, 
\[
\langle gZ^*(G), hZ^*(G)\rangle = \langle g_izZ^*(G), g_jwZ^*(G) \rangle,
\]
for some $z, w \in Z^*(G)$. 

From Lemma \ref{QuotientNilAdjacency} follows that $\langle g_izZ^*(G), \, g_jwZ^*(G)\rangle$ is a nilpotent subgroup of $G/Z^*(G)$ if and only if $\langle g_iz, g_jw\rangle = \langle f(gZ^*(G), z), f(hZ^*(G), w)\rangle$ is a nilpotent subgroup of $G$, 
so, by restricting $f$ to $\widetilde{G} - Z^*(\widetilde{G})$, one obtains a bijective function preserving adjacency. That is an isomorphism between the corresponding graphs. 
Finally, since $Z^*(G)$ es nilpotent follows from Theorem \ref{KappaProductNilpotent} that $\kappa(G) = \kappa(G/Z^*(G)\times Z^*(G)) = \kappa(G/Z^*(G))$. 
\end{proof}

\begin{lemma}
\label{ZetaDivisor}
$|Z^*(G)| = \gcd (k_1(G),\ldots, k_n(G), \,|G|)$, where $n = \kappa(G)$.   
\end{lemma}

\begin{proof}
Using Proposition \ref{NilDivisor} follows that  $|Z^*(G)|$ divides $k_i(G/Z^*(G)\times Z^*(G)) = k_i(G)$ for all $i \in \{1, \ldots, \kappa(G)\}$. Now, note that $\sum_{i = 1}^{n} k_i(G) + |Z^*(G)| = |G|$. 
Thus 
\[
1 = \tfrac{|G|}{|Z^*(G)|} - \sum_{i = 1}^{n}\tfrac{k_i(G)}{|Z^*(G)|},
\]
and so
\[
\gcd\left(\tfrac{k_1(G)}{|Z^*(G)|}, \ldots, \tfrac{k_n(G)}{|Z^*(G)|}, \tfrac{|G|}{|Z^*(G)|}\right) = 1,
\]
which implies the assertion.
\end{proof}

\begin{corollary}
\label{OrderGroups}
If $\gcd(k_i(G), k_j(G)) = 1$ for some $i, j$, then for any finite non-nilpotent group $H$ with $\Gamma_{\frakN}(G) \cong \Gamma_{\frakN}(H)$, then $|G| = |H|$.
\end{corollary}

\begin{proof}
Since $\Gamma_{\frakN}(G) \cong \Gamma_{\frakN}(H)$, it follows that $|G| - |Z^*(G)| = |H| - |Z^*(H)|$. From  Lemma \ref{ZetaDivisor}, $|Z^*(G)| = 1 = |Z^*(H)|$, and so $|G| = |H|$.
\end{proof}

\begin{theorem}\label{not-seft-comp}
No finite non-nilpotent group has a self-complementary nilpotent graph. 
\end{theorem}

\begin{proof}
Suppose that $\Gamma_{\frakN}(G)$ is a self-complementary graph. Then, there exists a bijective function $f: G - Z^*(G) \longrightarrow G - Z^*(G)$ such that $\langle x, y \rangle$ is nilpotent if and only if $\langle f(x), f(y) \rangle$ is non-nilpotent. Consequently, for any  $x \in G - Z^*(G)$ holds 
\[
 |\nil_G(x)| - |Z^*(G)| - 1 = \deg(x) = \deg(f(x)) = |G| - |\nil_G(f(x))|.
\]
Thus, by \cite[Teorema 4.1]{Kumar},  $|Z^*(G)|$ divides $|G| - |\nil_G(f(x))| = |\nil_G(x)| - |Z^*(G)| - 1$ and $|\nil_G(x)| - |Z^*(G)|$ for all $x \in G- Z^*(G)$, so $|Z^*(G)| =1$. 

Let  $C_2$ be the cyclic group of order 2, $H := G \times C_2$, and consider the function $g: H - Z^*(H) \longrightarrow H - Z^*(H)$ defined by $g(x,n) = (f(x),n)$.
It is easy to see that this function is well-defined and bijective. 
It follows from Lemma  \ref{ComponentsNilAdjacency} that $\langle (x, n), (y, m) \rangle$ is nilpotent if and only if $\langle (f(x), n), (f(y), m) \rangle = \langle g(x, n), g(y, m) \rangle$ is non-nilpotent, for all $(x, n), (y, m) \in H - Z^*(H)$. Therefore, the nilpotent graph of $H$ is self-complementary too.  By a similar argument to the one used at the beginning, $|Z^*(H)| = 1$, which is a contradiction (because $|Z^*(H)| = |Z^*(G) \times C_2| = 2$, by Corollary \ref{NilProduct}).
\end{proof}

\section{Nilpotent graph of the Dihedral group}
\label{Section5}

In this section, a formula for the number of the connected components of the non-nilpotent dihedral group $D_n$ is presented.\\

It is well-known that $Z(D_n) \neq 1$ if and only if $n$ is even, in other words, no reflection of $D_n$ commutes with any rotation different from the identity if and only if $n$ is odd. This fact will be used in Theorem \ref{kDn}. 

\begin{theorem}\label{kDn}
Let $n \in \mathbb{Z}_{>0}$. Then 
\[
\kappa(D_n) =\begin{cases}     n + 1 & \mbox{ if } 2\nmid n\\
                               m + 1 & \mbox{ if } 2\mid n \wedge n \mbox{ is not a power of } 2
              \end{cases} 
\]
where $n = 2^km$,  $m \geq 3$ and $2\nmid m$.
\end{theorem}

\begin{proof}
\begin{enumerate}
    \item Suppose $n$ is an odd number. Let $\langle r \rangle$ be the subgroup of rotations of $D_n$ and $s$ be a reflection. Since $\gcd(|r^i|, |r^js|) = 1$ for all $i\in \{1, \ldots, n-1\}$, $j \in \{0, \ldots, n\}$, then $\langle r^i, r^js \rangle$ is non-nilpotent for all $i\in \{1, \ldots, n-1\}$, $j \in \{0, \ldots, n-1\}$. Otherwise, it would follow from Theorem \ref{NilpotentEquivalence} that some rotation different from the identity commutes with some reflection, which is not possible since $n$ is odd.\\
    
    On the other hand, if $i\neq j$, then $\langle r^is, r^js \rangle$ contains at least one rotation different from the identity, again implying that some reflection commutes with some rotation. Therefore $\langle r^is, r^js \rangle$ is nilpotent if and only if $i = j$. Then $\Gamma_{\frakN}(D_n)$ has one component with $n-1$ vertices and $n$ components with one element. That is, $\kappa(G) = n+1$.\\
    
    \item   From Theorem \ref{GraphIsomorphism} follows that $\kappa(D_n) = \kappa(D_n/Z^*(D_n))$. Thus, by Example \ref{NilpotentizerDihedral}, $\kappa(D_n/Z^*(D_n)) = \kappa(D_m)$. Since $m$ is an odd number, by part $1$, $\kappa(D_n) = \kappa(D_m) = m + 1$.
\end{enumerate}

\end{proof}


\begin{example}\label{exnilDn}
If $n=5$ or $n=12=2^{2}\cdot m$, the nilpotent graph of $D_n$ has $n+1=6$ or $m+1=3+1=4$ connected components (by Examples \ref{nilD5} and \ref{nilD12}), respectively, which illustrates Theorem \ref{kDn}.
\end{example}

\section{Nilpotent graph of some projective special linear groups}
\label{Section6}
If $C$ is a connected component of the nilpotent graph of $G$; in this section, necessary and sufficient conditions for the union of $C$ and the hypercenter of $G$ to contain one of its Sylow $p$-subgroups are studied. If $q$ is a power of a prime number, a formula for the number of connected components of the nilpotent graph of $\psl(2,q)$ is presented. Finally, it is proved that the nilpotent graph of a group with nilpotentizer of even order is never Eulerian.

It is easy to see that  $G$ acts by conjugation on $\mathcal{C}(G)$  as follows $g\cdot C := gCg^{-1}$. Lemma \ref{ComponentNormalizer} shows that $C \subseteq N_G(C)$. 

\begin{lemma}
\label{ComponentNormalizer}
Every connected component of $\Gamma_{\frakN}(G)$ is invariant under conjugation by its element.

\end{lemma}

\begin{proof}
Let $g \in C$, $y \in g^{-1}Cg$ be arbitrary. Then there exist $x, h_1, \ldots, h_n \in C$ such that $y = g^{-1}xg$ and $\langle g, h_1 \rangle, \langle h_1, h_2 \rangle,\ldots, \langle h_n, x\rangle$ are nilpotent subgroups of $G$, because $C\in \mathcal{C}(G)$. Then, 
$\langle g, g^{-1}h_1g \rangle$, $\langle g^{-1}h_1g, g^{-1}h_2g \rangle,\ldots,$ $\langle g^{-1}h_ng, g^{-1}xg \rangle = \langle g^{-1}h_ng, y\rangle$ are nilpotent subgroups of $G$. That is, a path exists from $g$ to $y$ so that $y \in C$. Therefore $g^{-1}Cg \subseteq C$ and since $|g^{-1}Cg| = |C|$, then $g^{-1}Cg = C$.
\end{proof}

\begin{remark}
Note that $\stab_G(C) = N_G(C) = N_G(C\cup Z^*(G))$ for all $C \in \mathcal{C}(G)$, so that $\stab_G(C)$ acts by conjugation on $C \cup Z^*(G)$. Moreover, $\stab_{N_G(C)}(g) = C_{N_G(C)}(g)$ for all $g \in C \cup Z^*(G)$. 
\end{remark}

If $C \in \mathcal{C}(G)$, Theorem \ref{psylow-component} provides necessary and sufficient conditions for the set $C \cup Z^*(G)$ to contain a Sylow $p$-subgroup of $G$.   

\begin{theorem}\label{psylow-component}
Let $C \in \mathcal{C}(G)$, and $p$ be a prime divisor of $|G|$.

\begin{enumerate}
    \item If $P$ is a Sylow $p$-subgroup of $G$ contained in $C \cup Z^*(G)$, then  $p \mid |C \cup Z^*(G)|$.\\

    \item If $p\mid gcd(|C|+1,|N_G(C)|)$ and $Z^*(G) = 1$, then $C\cup 1$ contains a Sylow $p$-subgroup of $G$.
\end{enumerate}

\end{theorem}

\begin{proof}

\begin{enumerate}
    \item  If $P \cap Z^*(G) \neq 1$, then $Z^*(G)$ contains an element of order $p$, i.e. $|Z^*(G)|$ is divisible by $p$. Since $|Z^*(G)| \mid |C|$, by Lemma \ref{ZetaDivisor}, then 
$|C\cup Z^*(G)| = |C| + |Z^*(G)| = |Z^*(G)|m = pkm$, for some integer numbers $k$ and $m$. 
Suppose now  that $P\cap Z^*(G) = 1$. Since $P \subseteq C \cup Z^*(G)$, there exists $x \in C$ such that $\ord(x) = p$. Thus $x \in N_G(C)$ (by Lemma \ref{ComponentNormalizer}) and so $\langle x \rangle$ acts by conjugation on the elements of $C \cup Z^*(G)$. 
Note that, for $j \in \{1, \ldots, p-1\}$ and $g\in G$, $g^{x^j} = g$ if and only if $g \in C_G(x)$. In addition, if $g \in C_G(x)$, then $\langle x, g \rangle $ is nilpotent, so that $C_G(x) \subseteq C\cup Z^*(G)$. 
Thus
\[
p = \ord(x) = |\orb_x(g)| \ |\stab_x(g)|,
\] 
for all $g \in C \cup Z^*(G)$. Then $|\orb_x(g)| = p$ or $|\orb_x(g)| = 1$ for all $g \in C \cup Z^*(G)$. Thus, since $|\orb_x(g)| = 1$ if and only if $g \in C_G(x)$, 

\begin{eqnarray*}
|C\cup Z^*(G)|  &=& \sum_{i=1}^n |\orb_x(g_i)| =|C_G(x)| + pn_2     
\end{eqnarray*}

where $n$ and $n_2$ are the number of orbits, and the number of orbits of size $p$, respectively. Hence, since $\langle x \rangle \leq C_G(x)$, $p  \mid |C\cup Z^*(G)|$.\\

\item Suppose that no Sylow $p$-subgroup of $G$ is contained in $C\cup 1$. Since $p \mid |N_G(C)|$, then there exists an $x \in N_G(C)$ such that $\ord(x) = p$ and so $\langle x \rangle$ acts on $C\cup 1$ by conjugation. Then, $g^y \neq g$ for all $1\neq y \in \langle x \rangle$, and $g \in C$; otherwise, $\langle y, g\rangle$ is nilpotent, and so $y \in C \cup 1$, which is a contradiction. Thus, $y \in \langle x \rangle$ stabilizes $g \in C \cup 1$ if and only if $g = 1$ or $y = 1$. Hence, by the Burnside's Lemma, 

\begin{eqnarray*}
p \cdot |(C \cup 1)/\langle x \rangle| &=& \ord(x) \cdot|(C \cup 1)/\langle x \rangle|=\sum_{y \in \langle x \rangle}|\fix(y)|\\
                                 &=& |C \cup 1| + (\ord(x)-1)=|C|+p. 
\end{eqnarray*}

Thus $p\mid |C|$, which is a contradiction. Therefore, a Sylow $p$-subgroup of $G$ must be contained in $C \cup 1$.

\end{enumerate}
\end{proof}

For the rest of this section, $G$ denotes the projective special linear group $\psl(2,q)$,  where $q = p^f$ for some prime number $p$. In addition, $k := \gcd(q-1,2)$, $r :=\frac{q-1}{k}$, and $t :=\frac{q+1}{k}$.

\begin{theorem}\cite[Chapter II, S\"atze 8.2, 8.3, 8.4]{Huppert} \label{8.2}
\begin{enumerate}
\item Let $P$ be a Sylow $p$-subgroup of $G$. Then $P$ is isomorphic to the additive group of the finite Field $\F_q$, an elementary abelian group of order $q$. Further, $G$ has exactly $q+1$ Sylow $p$-subgroups. 
\item  The subgroup $U$ of $G$, which fixes the values $0$ and $\infty$, is cyclic of order $r$.
\item  $G$ has a cyclic subgroup $S$ of order $t$.
\end{enumerate}
\end{theorem}


 

\begin{theorem}\cite[Chapter II, Satz 8.5]{Huppert} \label{8.5}\\
Let $P,U,S$ be as in Theorem \ref{8.2}, and
\[
\mathcal{P} := \{P^x, U^x, S^x : x\in G\}.
\]
Then, each element of $G$ different from 1 lies in exactly one of the groups from $\mathcal{P}$ (i.e., $\mathcal{P}$ forms a partition of $G$).
\end{theorem}



Thus, by Theorem \ref{8.5},  every element of $G$ different from $1$ belongs to only one of the conjugacy classes of the subgroups $P$, $U$, $S$, described in Theorem \ref{8.2}. Moreover, $N_G(P) \cong \Z_p^f \rtimes \Z_r$, $N_G(U) \cong D_r$ y $N_G(S) \cong D_t$.\\

Let $C_P$, $C_U$, $C_S$ denote the connected component of $\Gamma_{\frakN}(G)$ containing elements of $P$, $U$ and $S$, respectively.

\begin{theorem}\label{OrbitsGPSU}
The following statements hold:
\begin{enumerate}
    \item If $q>3$, then $|\mathcal{C}(G)/G| = 3$.
    \item  $|\orb_G(C_P)| = q+1$. 
    \item  If $q \geq 13$, then $|\orb_G(C_U)| = 1$ if and only if $q \equiv 1 \bmod 4$.
    \item If $q > 3$ and $q \neq 7, 9$, then $|\orb_G(C_S)| = 1$ if and only if $q \equiv 3 \bmod 4$.
\end{enumerate}
\end{theorem}

\begin{proof}

\begin{enumerate}
    \item Let $1\neq x \in G$ and suppose that $\langle x, y\rangle$ is nilpotent for some $y \in G - Z^*(G)$. If $x \in P$, then $y\in P^g$ for some $g\in G$. Otherwise, $y\in U^h$ or $y\in S^z$, for some $h, z\in G$, so that $x$ and $y$ commute, because their orders are relative primes. However, $xy$ must be in some conjugate of $P$, $U$, or $S$, so the order of one of these groups must be divisible by $p$ and by some other prime that divides $r$ or $t$, which is false. Therefore $y \in$ $P$. Analogously, if $x \in U$, then $y$ must be in some conjugate of $U$, and if $x \in S$, then $y$ must be in some conjugate of $S$. Therefore, $C_U, C_S \notin \orb_G(C_P)$ and $C_U \notin \orb_G(C_S)$, and since each orbit contains components with elements of one of these groups, then $|\mathcal{C}(G)/G| = 3$.\\

    \item Suppose $|\orb_G(C_P)| \neq q+1$. Necessarily $|\orb_G(C_P)| < q+1$, since $|\orb_G(C_P)$, by part $1$, contains only components with elements in the conjugates of $P$. Then, there must exist $x \in P$ and $y$ in some conjugate of $P$, say $P'$, such that $\langle x, y \rangle$ is nilpotent. Consequently there exists $z \neq 1$ such that $z \in Z(\langle x, y \rangle)$, so that $z \in N_G(P) \cap N_G(P')$, since $x \in P^z\cap P$, $y \in P'^z \cap P'$ and $x \neq 1 \neq y$. Moreover, $\langle x, z\rangle$ is also nilpotent, so $z$ must be in some conjugate of $P$. Since $P$ is elementary abelian, then $|z| = p$, so, $z \in P \cap P'$, since the only subgroups of order divisible by $p$ of $N_G(P)$ and $N_G(P')$ are $P$ and $P'$. But this is impossible since $P \cap P' = 1$. Therefore, $C_P = P- 1$ so $|\orb_G(C_P)| = |G : N_G(P)| = q+1$.\\

    \item Suppose $|\orb_G(C_U)| = 1$. Then $C_U$ has elements of $U$ and some conjugate of $U$, say $U'$. That is, there exist $x \in U$, $y \in U'$, $x \neq 1 \neq y$ such that $\langle x, y\rangle$ is nilpotent. Thus, there exists $z \in \langle x, y\rangle$ such that $z \neq 1$ and commutes with $x$ and $y$. Therefore, $z \in N_G(U) \cap N_G(U')$, since $x \in U^z\cap U$, $y \in U'^z\cap U'$ and $x \neq 1 \neq y$.
    On the other hand, if $z \in U$, then $z \notin U'$, since $U \cap U' = 1$ and $z \neq 1$, so that $\ord(z) = 2$, due to $z \in N_G(U') \cong D_{\frac{q-1}{k}}$. Analogously, if $z \in U'$ o $z \notin U \cup U'$,  $\ord(z) = 2$. Since $z$ must also be in some conjugate of $U$ by part $1$, it follows that $2 = \ord(z)$ divides $|U| = \frac{q-1}{k}$, so $k = 2$ and hence $q \equiv 1 \bmod 4$. Conversely, suppose that $q \equiv 1 \bmod 4$. Then $k = 2$ and $\frac{q-1}{2}$ is a even number. Note that $N_G(U) \leq N_G(C_U)$, and since $N_G(U) \cong D_{\frac{q-1}{2}}$ and $\frac{q-1}{2}$ is an even number, $Z(N_G(U)) \neq 1$, so there exists $z \notin U$ such that $z$ commutes with some element $x \in U$. Thus $\langle x, z\rangle$ is nilpotent, and by part $1$, $z$ must belong to some conjugate of $U$. Then $U^g-1 \subseteq C_U$, for some $g \in G - N_G(U)$, which implies, by  Lemma \ref{ComponentNormalizer}, that $N_G(C_U)$ contains at least two cyclic groups of order $\frac{q-1}{2}$. Hence $N_G(U) < N_G(C_U)$. In addition, since $q \geq 13$,  $N_G(U)$ is a maximal subgroup (by \cite[Table 5]{Maximal}), so that  $N_G(C_U) = G$ and $|\orb_G(C_U)| = |G : N_G(C_U)| = 1$.\\

    \item Similarly to the previous proof, we use the fact that $N_G(S)$ is maximal when  $q \neq 7, 9$ \cite[Table 5]{Maximal}.
    
\end{enumerate}
\end{proof}

\begin{table}[H]
\label{ccPSL}
\begin{center}
\begin{tabular}{|c|c|c|c|c|c|c|c|c|} 
 \hline
 $q$ & 2 & 3 & 4 & 5 & 7 & 8 & 9 & 11 \\ [1ex] 
 \hline 
 $\kappa(G)$ & 4 & 5 & 21 & 21 & 37 & 73 & 47 & 79\\ [1ex] 
 \hline
\end{tabular}
\caption{$\kappa(\psl(2,q))$ for $q \leq 11$.}\label{tab1}
\end{center}
\end{table}

\begin{theorem}
Let $q \neq 2, 3, 5$. Then 
\[
\kappa(G) =  \begin{cases}
    q^2 + q + 1 &  \text{if} \ q \equiv 0 \bmod 4\\
    q + \frac{(q-1)q}{2} + 2 &  \text{if} \ q \equiv 1 \bmod 4\\
     q + \frac{(q+1)q}{2} + 2 &   \text{if} \ q \equiv 3 \bmod 4
\end{cases}    
\]
\end{theorem}

\begin{proof}
If $q < 13$, $q \neq 2, 3, 5$, the statement is true (by Table \ref{tab1}). If $q \geq 13$, there are three cases:

\noindent\textbf{Case 1.} If $q \equiv 0 \bmod 4$, since $N_G(U)$ and $N_G(S)$ are maximal subgroups, see \cite[Table 5]{Maximal},  follows that $N_G(U) \leq N_G(C_U)$, and $N_G(S) \leq N_G(C_S)$. By Theorem \ref{OrbitsGPSU} (3) and (4), holds $|\orb_G(C_U)| \neq 1$ y $|\orb_G(C_S)| \neq 1$, then  $N_G(U) = N_G(C_U)$, $N_G(S) = N_G(C_S)$ and
\begin{align*}
 |\orb_G(C_U)| & = |G : N_G(U)| = \frac{(q+1)q}{2}, \\
 |\orb_G(C_S)| & = |G : N_G(S)| = \frac{(q-1)q}{2}.   
\end{align*}

On the other hand, by Theorem \ref{OrbitsGPSU} (1) and (2), $\left|\frac{\mathcal{C}(G)}{G}\right| = 3$ and $|\orb_G(C_P)| = q+1$, implying that 
\[\kappa(G) = |\orb_G(C_P)| + |\orb_G(C_U)| + |\orb_G(C_S)| = q^2 + q + 1.\]

\noindent\textbf{Case 2.} If $q \equiv 1 \bmod 4$, then by Theorem \ref{OrbitsGPSU} (3) and (4) follows
\begin{align*}
 |\orb_G(C_U)| & = |G : N_G(C_U)| = 1, \\
 |\orb_G(C_S)| & = |G : N_G(S)| = \frac{(q-1)q}{2}.   
\end{align*}
Therefore, $\kappa(G)  = q+2+\frac{(q-1)q}{2}$. 
 
\noindent\textbf{Case 3.} If $q \equiv 3 \bmod 4$, by the same argument, $\kappa(G) = q+2 + \frac{(q+1)q}{2}$.
\end{proof}

\begin{proposition}
The nilpotent graph of a finite non-nilpotent group, $G$ with nilpotentizer of even order, is non-Eulerian
\end{proposition}

\begin{proof}
Let $x \in G - Z^*(G)$. Then, deg$(x) = |nil_G(x)| - |Z^*(G)| - 1$. Since $|Z^*(G)| | |nil_G(x)|$ and $|Z^*(G)|$ is even, it follows that $|nil_G(x)|$ is even. Thus, deg$(x)$ is odd, which means the graph is non-Eulerian.
\end{proof}

\bibliographystyle{plain}

\end{document}